\documentclass[a4paper,11pt]{amsart}
\usepackage{graphicx} % Required for inserting images
\usepackage{amsfonts, amssymb, amsmath}
\usepackage{amsthm}
\usepackage{mathtools}
\usepackage[margin=1in]{geometry}
\usepackage{enumerate}
\usepackage[all]{xy}
\usepackage{color, xcolor}

\usepackage{tikz, tikz-cd}
\usetikzlibrary{3d}
\usepackage{caption}
\usepackage{subcaption}
\usepackage[all]{xy}

\newcommand{\PP}{\mathbb{P}}
\newcommand{\CC}{\mathbb{C}}
\newcommand{\FF}{\mathbb{F}}
\newcommand{\QQ}{\mathbb{Q}}
\newcommand{\ZZ}{\mathbb{Z}}

\newcommand{\si}{\sigma}
\newcommand{\Si}{\Sigma}

\newcommand{\cO}{\mathcal{O}}
\newcommand{\cM}{\mathcal{M}}
\newcommand{\cN}{\mathcal{N}}
\newcommand{\cY}{\mathcal{Y}}
\newcommand{\cB}{\mathcal{B}}
\newcommand{\cT}{\mathcal{T}}

\DeclarePairedDelimiter\prbra{(}{)}

\theoremstyle{definition}

\newtheorem{Thm}{Theorem}[section]
\newtheorem{Prop}[Thm]{Proposition}

\newtheorem{Rmk}[Thm]{Remark}
\newtheorem{Cor}[Thm]{Corollary}
\newtheorem{Ex}[Thm]{Example}

\usepackage{epic}

\begin{document}

\title{Deformation of pairs of $\PP^3$ and hypersurfaces}
\author{Jungkai Chen, Yongnam Lee and Phin-Sing Soo}

\address{Department of Mathematics, National Taiwan University, No. 1, Sec. 4, Roosevelt Rd., Taipei 10617, Taiwan}
\email{jkchen@ntu.edu.tw}

\address{Center for Complex Geometry, Institute for Basic Science (IBS), 55 Expo-ro, Yuseong-gu, Daejeon, 34126 Korea}
\email{ynlee@ibs.re.kr}

\address{Graduate School of Mathematical Sciences, The University of Tokyo, 3-8-1 Komaba, Meguro, Tokyo, 153-8914 Japan}
\email{pssoo@g.ecc.u-tokyo.ac.jp}

\subjclass{14J10, 14J70}
\keywords{deformation, hypersurface, canonical singularity}

\date{\today}

\begin{abstract}
Motivated by DeVleming's work on moduli of surfaces in $\PP^3$ and Chen-Hu-Jiang's work on moduli of threefolds with volume $2$ and geometric genus $4$, we study the deformation of pairs of $\PP^3$ and hypersurfaces using the classification of $\QQ$-Gorenstein degenerations of $\PP^3$ with canonical singularities. We prove that if a degenerating threefold has canonical singularities, then the moduli space is smooth at the corresponding pair. Consequently, we find some boundary divisors of the moduli of smooth hypersurfaces. Finally, using the double cover method, we derive some information on the moduli space of threefolds $X$ with canonical singularities with the same volume and geometric genus as a double cover of $\PP^3$ branched over a hypersurface.
\end{abstract}

\maketitle

\section{Introduction}
In this paper, we work over the field of complex numbers. Let $Y$ be a normal projective variety. Suppose that $Y$ is a $\QQ$-Gorenstein degeneration of $\PP^3$, that is, there exists a flat projective  morphism $\pi: \cY\to\Delta$ over a DVR with the central fiber $Y$ such that $\pi^{-1}(t)=Y_t\cong \PP^3$ for $t\ne 0$ and $K_{\cY/\Delta}$ is $\QQ$-Cartier. If $Y$ has at worst terminal singularities, then $Y\cong \PP^3$ (cf. \cite[Theorem 4.33]{DeV22}) \cite[Corollary 1.4]{HP24}). However, $Y$ is not necessarily isomorphic to $\PP^3$ if $Y$ has canonical singularities. Due to H\"oring and Peternell \cite[Theorem 1.3]{HP24}, normal projective canonical degenerations $Y$ of $\PP^3$ are classified; they belong to one of the following possible cases. 
\begin{itemize}
    \item[a)] (type I) $Y\cong \PP^3$.
    \item[b)] (type II) Let $T=\mathbb F_0$ and $V=\PP(\cO_T\oplus\cO_T(-2, -2))$, then we have $p: V\to Y$ given by the contraction to the cone.
    \item[c)] (type III) $Y \cong \PP(1,1,2,4)$, which is a cone over $S=\PP(1,1,2)$. 
    Another detailed description is as follows. Let $T=\mathbb F_2$ and $V=\PP(\cO_T\oplus\cO_T(-2\si-4\ell))$ where $\si$ is the negative section and $\ell$ is a fiber in $T$. Then  $p: V\to Y$ is obtained by weighted blowup $\frac{1}{4}(1,1,2)$ and then blowup the $1$-dimensional $A_1$ singularity.
    \item[d)] (type IV) Let $Y'\in |2\zeta-p^*12 \ell|$ on the projective bundle $p: \PP(V)\to T=\FF_4$ for $V = \cO_T(12 \ell + 2 \sigma) \oplus \cO_T(6 \ell) \oplus \cO_T$. Here, $\zeta$ is the tautological bundle on $\PP(V)$; $\sigma$ is the negative section, and $\ell$ is the fiber in $T$.
    Then the fibration $\varphi: Y'\to T$ is a conic bundle that has double fibers over $\sigma$ and smooth fibers elsewhere, and $Y'$ has canonical singularities along a section $\pi^{-1}(\sigma)$. The morphism $\phi: Y'\to Y$ to the anticanonical model $Y \subset \PP^{34}$ contracts exactly the $\pi$-section $E_1$ corresponding to $V\to\cO_T$. Moreover, $\phi^*(-K_{Y})=-K_{Y'}$.
\end{itemize}

The type IV case is also described in \cite[Section 4.2]{ADL23}. We note that the threefold $Y$ of type IV has a unique canonical Gorenstein singularity, but it is not $\QQ$-factorial. 
In fact, $L$ in $Y$ has Cartier index 4, where $L$ is the degeneration of $\cO_{\PP^3}(1)$. 
We describe the relation between the two constructions and a resolution of type IV in Section~\ref{sec:type IV}. 

Let $\cN_d$ be the moduli space of pairs $(\PP^3, B_d)$ where $B_d$ is a smooth hypersurface of degree $d$. For $d\ge 5$, DeVleming \cite{DeV22} showed that $\cN_d$ can be compactified by the moduli space $\overline{\cN}_{(d, 4)}$ of the pairs $(Y, B_{d, Y})$ satisfying the following conditions:
\begin{itemize}
    \item[(i)] the pair $(Y, (4/d+\epsilon)B_{d, Y})$ is slc and the divisor $K_Y + (4/d + \epsilon)B_{d, Y}$ is ample for some  $\epsilon >0$;
\item[(ii)] the divisor $dK_Y + 4B_{d, Y}$ is linearly equivalent to zero;
\item[(iii)] there is a deformation $(\cY, \cB_{d, \cY})/\Delta$ of $(Y, B_{d, Y})$ over the germ of a curve such that the general
fiber $Y_t$ is smooth, $Y_t$ admits a smooth deformation to $\PP^3$, and the divisors $K_{\cY/\Delta}$ and $\cB_{d, \cY}$ are
$\QQ$-Cartier.
\end{itemize}

She proved that if $(Y, B_{d, Y})\in\overline{\cN}_{(d, 4)}$ with canonical singularities and $d$ is odd then $Y\cong\PP^3$ or isomorphic to the cone over a quadric surface \cite[Theorem 1.3]{DeV22}. In addition, she shows that these pairs form a divisor in $\overline{\cN}_{(d, 4)}$ \cite[Proposition 3.10 and Proposition 4.40, 4.41]{DeV22}.

When $d=2d_1$ is even, object in $\cN_d$ can be understood as a threefold which is a double cover of $\PP^3$ branched over $B_d$. More precisely, since $B_d$ is assumed to be smooth, we have the isomorphism between the moduli space of  the pairs $(\PP^3, B_d)$ and the moduli space of  the pairs $(\PP^3, \frac{1}{2}B_d)$, and the latter one is isomorphic to the moduli of threefolds which are double cover of $\PP^3$ branched over $B_d$.

Let $p: X\to Y=\PP^3$ be a double cover branched over $B_d$. Since $$K_X=p^*(K_{\PP^3}+\frac{1}{2}B_d),$$
we have ${\rm Vol}(X)=2(d_1-4)^2$ and $p_g(X)=\frac{1}{6}(d_1-1)(d_1-2)(d_1-3)$. Let $\cM_d^\mathrm{can}$ (resp. $\cM_d^\mathrm{sm}$)  be the moduli space of threefolds $X$ with canonical singularities (resp. without singularities) satisfying ${\rm Vol}(X)=2(d_1-4)^2$ and $p_g(X)=\frac{1}{6}(d_1-1)(d_1-2)(d_1-3)$. Through this double cover correspondence, $\cN_d\subseteq \cM_d^\mathrm{can}$ if $d$ is even. We note that $X$ has semi log canonical singularities if and only if $(Y, \frac{1}{2}B_{d, Y})$ has semi log canonical singularities (cf. \cite[Proposition 3.16]{Kol97}). The moduli space $\cM_d^\mathrm{can}$ can be compactified by the KSBA moduli space $\overline{\cM}_d^\mathrm{KSBA}$. We refer to \cite{Kollar} for the KSBA moduli spaces.

\medskip
In this paper, we study the boundaries of $\overline{\cN}_{(d, 4)}$ and $\cM_d^\mathrm{can}$ when $d$ is even through the $\QQ$-Gorenstein degeneration of $\PP^3$ with canonical singularities. 

First, we prove the following theorem.
\begin{Thm}
The moduli space $\overline{\cN}_{(d, 4)}$ is smooth at the pair $[(Y, B_{d, Y})]$ if $Y$ has canonical singularities.
\end{Thm}
When $Y$ is of type I, II, or III,  this theorem is also mentioned in \cite[Proposition 3.10]{DeV22} and \cite[Lemma 5.4]{ADL23}. The new ingredient is the type IV case, which will be proved in Section~\ref{sec:type IV}.

\medskip 

In Section~\ref{sec:type IV}, by using the idea in \cite{DeV22} and the description of the type IV, we show that $(Y, B_{d, Y})$ cannot be in $\overline{\cN}_{(d, 4)}$ if $d\not\in 4\ZZ$. This is shown in \cite{DeV22} when $d$ is odd.

As a corollary, we find some boundary divisors of the moduli of smooth hypersurfaces. For instance, we have the following result. This is proved in Proposition~\ref{surface in type II} and Corollary~\ref{surface in type IV}.

\begin{Cor}
    Let $d=2d_1\in 4\ZZ$. Then we find two smooth boundary divisors via the type II and IV degeneration of $\PP^3$.
    \begin{enumerate}
    \item By considering the type II, there is a smooth boundary divisor whose general element is a smooth complete intersection of multi-degrees $(2, d_1)$ in the weighted projective surface $\PP(1,1,1,1,2)$
    \item By considering the type IV, there is a smooth boundary divisor whose general element is a smooth surface with a pencil structure.
    \end{enumerate}
\end{Cor}

\medskip

We call a threefold of Type II (resp. Type III) if it is a double cover of the type II (resp. the type III) threefold $Y$ branched over $B_{d, Y}$.
Using the toric description of the resolution of the type II or III, along with the exact sequence of tangent sheaves between the double cover threefold and the threefold below, we prove the following result in Section~\ref{sec:Type_II} and Section~\ref{sec:type_III}.

\begin{Thm} Let $d=2d_1$ be even.
\begin{enumerate} 
    \item Type II gives a boundary divisor in the moduli space $\cM_d^\mathrm{can}$ and the smoothness of $\cM_d^\mathrm{sm}$ extends this boundary. 
    \item If $d_1$ is even then the moduli space $\cM_d^\mathrm{can}$ at Type III is smooth. 
\end{enumerate}
\end{Thm}

\medskip

\subsection*{Acknowledgements}
J. Chen is supported by National Science and Technology Council (114-2811-M-002-168). Y. Lee is supported by the Institute for Basic Science
(IBS-R032-D1). Y. Lee would like to thank the Department of Mathematics and NCTS in National Taiwan University for the hospitality during his visit in 2025 and 2026. P.-S. Soo is supported by the University of Tokyo Fellowship and would like to thank his advisor, Keiji Oguiso, and his lab members for their guidance and valuable suggestions.
The authors would like to thank Kenneth Ascher and Kristin DeVleming for some useful discussion and comments.

\section{Double cover over the degeneration of $\PP^3$ to the cone}\label{canonical}

Due to \cite[Theorem 1.6]{CHJ24}, the canonical model of smooth projective threefold of general type $W$ with $p_g(W) \ge 4$ and ${\rm Vol}(W)=2$ belongs to either of the following 2 types. 
\begin{itemize}
    \item[(a)] $W$ is a hypersurface in $\PP(1,1,1,1,5)$ defined by a weighted homogeneous polynomial $f$ of degree 10, where $f(x_0, \ldots, x_4)=x^2_4 + f_0(x_0, x_1, x_2, x_3)$ in suitable homogeneous coordinates $[x_0: \cdots : x_4]$ of $\PP(1, 1, 1, 1, 5)$.
    \item[(b)] $W$ is a subvariety in $\PP(1, 1, 1, 1, 2, 5)$ defined by 2 weighted homogeneous polynomials $q$ and $f$ of degrees 2 and 10 respectively, where
    $$q(x_0, x_1, x_2, x_3)=x^2_3 + q_0(x_0, x_1, x_2), \quad f(x_0,\ldots, x_5)=x^2_5 +f_0(x_0, \ldots,x_4)$$ 
    in suitable homogeneous coordinates $[x_0 : \cdots : x_5]$ of $\PP(1, 1, 1, 1, 2, 5)$. This is a specialization of canonical 3-folds of type I;
\end{itemize}
In each case, $W$ is simply connected. The moduli space $\cM_{{\rm Vol}=2, p_g=4}$ ($=\cM_{10}^\mathrm{can}$ in our notation) is irreducible, unirational, and ${\rm dim}~\cM_{2,4}=270$.

\medskip

In \cite{Soo24}, Soo shows that the dimension of moduli of threefolds in (b) is 269, so it provides a boundary divisor in $\cM_{10}^\mathrm{can}$. He also describes the threefolds in (b) in more detail, particularly their minimal models. We summarize the description for the convenience of the reader. Let $V=\cO_T\oplus \cO_T(-2, -2)$ and $Y'=\PP(V)$ where $T=\mathbb F_0$, and let $\pi: Y'\to T$. Let $E$ be the $\pi$-section corresponding to the quotient line bundle $V \to \cO_T(-2, -2)\to 0$. Then a general $B_0\in |5E+\pi^*(\cO_T(10, 10))|$ is a smooth divisor such that $B_0\cap E=\emptyset$ by Bertini's theorem.  We have a double covering $p: X'\to Y'$ branched over $B:=B_0+E$. Then $X'$ is a smooth threefold with $({\rm Vol}, p_g)=(2, 4)$, but not minimal. The canonical model $X$ of $X'$ is obtained by contracting $p^{-1}(E)$. Let $\tau: Y'\to Y$ be the contraction of $E$ in $Y'$. Then $X\to Y$ is a double covering branched over $\tau_*B_0$ and the isolated canonical singularity of $Y$. 
The surface $T=\mathbb F_0$ has an anticanonical embedding into $\PP^3$ as a smooth quadric and hence can be degenerated to a singular quadric $T_o$ in $\PP^3$, which can also be obtained by contracting the $(-2)$-curve in the second Hirzebruch surface $\FF_2$. Let $Y_o$ be the contraction of the section $\PP(\cO_{T_o})$ in $\PP(\cO_{T_o}\oplus \cO_{T_o}(-K_{T_o}))$. Then $Y_o$ is a degeneration of the above $Y$, and $Y_o$ has singularities along a curve. Then the double cover of $Y_o$ has canonical singularities along a curve \cite{Soo24}.

\begin{Ex}\label{type II and III}
This degeneration can be generalized to any $d_1 \ge 5$. We give a more detailed description as follows. Let $\mathcal{Y} \subset \mathbb{P}(1^4, 2) \times \mathbb{A}^2 \eqqcolon \mathcal{Z}$ be defined by $ty +\lambda x_3^3+x_2^2+x_1^2+x_0^2$, which we regard as a family over a two-parameter base $\mathbb{A}^2$ in coordinates $t$ and $\lambda$. 
Let $\mathcal{D}$ be the divisor on $\mathcal{Z}$ defined by $f(x_0,x_1,x_2,x_3)$ some general homogeneous polynomial of degree $2d_1$. 
Consider a double cover $\mathcal{X}$ along $\mathcal{Y}$ branched over $\mathcal{B} \coloneqq \mathcal{D}|_\mathcal{Y}$. More precisely, since $\mathcal{Y}$ is singular along $([0:0:0:0:1], t=0) \times \mathbb{A}^1_\lambda$, we first take the double cover of $\mathcal{Y}^\circ \coloneqq \mathcal{Y}|_{t \neq 0}$ and then compactify.  

% (\textbf{Question.} Does this make sense? At least $\mathcal{B}$ is even Cartier away from the central fiber, so maybe what we do is take the double cover on $\mathcal{Y}|_{t \neq 0}$ and then extend to $\mathcal{Y}|_{t=0}$?)
% 
\begin{itemize}
\item If $t \ne 0$, then $Y_t \cong \mathbb{P}^3$. Since $Y_t$ avoids the cone point in $\PP(1^4,2)$, $B_t \in |\cO_{Y_t}(2d_1)|$ is an even Cartier divisor and $X_t$ is a double cover of $Y_t$ branched along $B_t$. This is Type I example.

\item If $t=0$ and $\lambda \ne 0$, then $Y_\lambda$ is a quadric in $\mathbb{P}(1^4,2)$. 
In fact, since $Z:=\mathbb{P}(1^4,2)$ is a cone over $\mathbb{P}^3$, $Y_\lambda$ is a cone over a smooth quadric $Q \cong \mathbb{F}_0$ defined by $\lambda x_3^3+x_2^2+x_1^2+x_0^2$ in $\mathbb{P}^3$. 
$X_\lambda$ is a double cover of $Y_\lambda$ along $B_\lambda$.

We describe the smooth model. 
Blow up $Z$ along the singular point, we get $\phi: \tilde{Z} \to Z $ with exceptional divisor $\mathcal{E} \cong \mathbb{P}^3$, where $\tilde{Z}=\mathbb{P}_{\mathbb{P}^3}(\mathcal{O} \oplus \mathcal{O}(-2))$. 
Restricting to (preimage of) $Q$ induces 
$Y'_\lambda \cong 
\mathbb{P}_Q(\mathcal{O} \oplus \mathcal{O}(-2H_Q)) \to Y_\lambda$ 
with exceptional divisor 
$E \cong Q \cong \mathbb{F}_0$, where $H_Q = \sigma + \ell$ on $\mathbb{F}_0$. We also denote the induced map by $\phi: Y'_\lambda \to Y_\lambda$ by abuse of notation.
\[
\begin{tikzcd}
    Y'_\lambda \ar[r, hook] \ar[d, "\pi"']
    \ar[rrr, bend left, "\phi"]
    & \tilde{Z} \ar[r, "\phi"] \ar[d] & Z \ar[ld, dashed] 
    & \ar[l, hook'] Y_\lambda
    \\
    \mathbb{F}_0  \ar[r, hook, "|\sigma + \ell|"]
    & \mathbb{P}^3
\end{tikzcd}
\]
Let $H_{Y_\lambda}$ be the restriction of the (non-Cartier) hyperplane divisor $\cO_Z(1)$ onto $Y_\lambda$: it has index $2$ at the cone point. Also, $2H_{Y_\lambda}$ has multiplicity $1$ at the cone point, so we have 
\[
\begin{aligned}
    \phi^* 2H_{Y_\lambda} &= \pi^* 2H_Q + E. 
\end{aligned}
\]
In particular, we have 
\[
\phi^*B_\lambda = \phi^* 2d_1 H_{Y_\lambda}
\sim \pi^* 2d_1 H_Q + d_1 E
= d_1 E + 2d_1 \pi^*(\sigma + \ell).
\]
We round this up to an even divisor by adding an extra $E$ to take the double cover $X'_\lambda$ of $Y'_\lambda$ branched along $(d_1 + 1)E + 2d_1 \pi^*(\sigma + \ell)$. we have the following commutative diagram:
\[
\begin{tikzcd}
    X'_{t=0, \lambda} \ar[r] \ar[d, "2:1"] 
    & X_{t=0, \lambda} \ar[d, "2:1"] \\
    Y'_{t=0, \lambda} \ar[r] & Y_{t=0, \lambda}
\end{tikzcd}
\]
But we note that there is no deformation from $X'_{t=0, \lambda}$ to $X_{t, \lambda}$. This is Type II example.

\item If $t=0$ and $\lambda = 0$, then $Y_0$ is a singular quadric in $\mathbb{P}(1^4,2)$. In fact, since $Z:=\mathbb{P}(1^4,2)$ is a cone over $\mathbb{P}^3$, $Y_0$ is a cone over a singular quadric $Q \cong \overline{\mathbb{F}}_2$ defined by $x_2^2+x_1^2+x_0^2$ in $\PP^3$. $X_0$ is a double cover of $Y_0$ along $B_0$.

Blow up $Z$ as in II and restrict to $\overline{\mathbb{F}}_2$. 
We get $\tilde{Y}_0 \cong \mathbb{P}_Q(\mathcal{O} \oplus \mathcal{O}(-2H_Q)) \to Y_0$ with exceptional divisor 
$E \cong Q \cong \overline{\mathbb{F}}_2$.
Further resolve $Q$ by blowing up the singular point. This blowup can be identified with the map $\mathbb{F}_2 \to \overline{\mathbb{F}}_2$ given by the linear system $\lvert {\sigma + 2\ell} \rvert$. After base change (normalization of $\tilde Y_0\times_{\overline{\mathbb F}_2}\mathbb F_2$), we get a projective bundle $Y_0' \cong \mathbb{P}_{\mathbb{F}_2}
(\mathcal{O} \oplus \mathcal{O}(-2H_{\mathbb{F}_2}))$, where $H_{\mathbb{F}_2} = \sigma + 2\ell$.

% Blowup $Z$ along the singular point, we get $\tilde{Z} \to Z $ with exceptional divisor $\mathcal{E} \cong \mathbb{P}^3$, where $\tilde{Z}=\mathbb{P}(\mathcal{O} \oplus \mathcal{O}(-2))$. It induces 
% $\tilde{Y}_0 \cong \mathbb{P}_Q(\mathcal{O} \oplus \mathcal{O}(-2H)) \to Y_0$ with exceptional divisor 
% $E \cong \bar Q \cong {\mathbb{F}}_2$, where $H$ corresponds to $\sigma+2l$ on $\mathbb{F}_2$. We have $D \sim \pi^* 5H$, where $\pi: \tilde{Y}_0 \to \bar Q$ is the natural map.
% This is the Type III example.
\[
\begin{tikzcd}
    Y'_0 \ar[r] \ar[d, "\pi"']
    & \tilde{Y}_0 \ar[r, hook] \ar[d]
    \ar[rrr, bend left, "\phi"]
    & \tilde{Z} \ar[r, "\phi"] \ar[d] & Z \ar[ld, dashed] 
    & \ar[l, hook'] Y_0
    \\
    \mathbb{F}_2 \ar[r, "|\sigma + 2\ell|"]
    & \overline{\mathbb{F}}_2  \ar[r, hook]
    & \mathbb{P}^3
\end{tikzcd}
\]
Similar to Type II, we take the double cover $X'_0$ of $Y'_0$ branched along $(d_1 + 1)E + 2d_1 \pi^*(\sigma + 2\ell)$. We also have the following commutative diagram:
\[
\begin{tikzcd}
    X'_{t=0, \lambda=0} \ar[r] \ar[d, "2:1"] 
    & X_{t=0, \lambda=0} \ar[d, "2:1"] \\
    Y'_{t=0, \lambda=0} \ar[r] & Y_{t=0, \lambda=0}
\end{tikzcd}
\]
And we note that there is a deformation from $X'_{t=0, \lambda=0}$ to $X'_{t=0, \lambda}$ via deformation from $\mathbb F_2$ to $\PP^1\times \PP^1$.
This is Type III example. % (\textbf{Details to be checked.})
\end{itemize}
\end{Ex}

\begin{Rmk} Observe that $Y_0$ is isomorphic to $\PP(1,1,2,4)$, which appears on the list in \cite{ADL23}. Indeed, consider the map $\PP(1,1,2,4) \to \PP(1^4,2) = Z$ given by the  embedding $[u_0: u_1: v: w] \mapsto [u_0^2: u_0 u_1: u_1^2: v: w]$. Then the image is defined by $x_0 x_2 - x_1^2$ in $Z$, which becomes the defining equation of $Q$ (and $Y_0$) after a coordinate change.
\end{Rmk}

The next proposition shows that if $Y$ is a $\QQ$-Gorenstein degeneration of $\PP^3$ and a cone over a Hirzebruch surface $\mathbb{F}_e$ or a rational normal scroll $\overline{\mathbb{F}}_e \cong \PP(1,1,e)$ then $Y$ should be the type II or the type III without assuming canonical singularities.

\begin{Prop}\label{cone}
Let $T=\mathbb F_e$ and $V=\PP(\cO_T\oplus \cO_T(-D))$ for some divisor $D$ on $T$. Let $p: V\to T$ the projection with $\Sigma$ the section corresponding to $\cO_T(-D)$.

\begin{enumerate}
    \item Suppose $D$ is ample, and let $f: V\to W$ be the morphism contracting $\Sigma$ to the cone $W$ so that $W$ is a $\QQ$-Gorenstein degeneration of $\PP^3$. Then $W$ is the type II.

    \item Suppose $D$ is nef and big but not ample, and let $f: V\to W$ be the composition of $f': V\to W'$ contracting $p^*\sigma$ and $\tau: W' \to W$ contracting $f'(\Sigma)$ to the cone $W$ so that $W$ is a $\QQ$-Gorenstein degeneration of $\PP^3$. Then $W$ is the type III.
\end{enumerate}
\end{Prop}

\begin{proof}
Let $D=a\si+b\ell$ where $\si$ is the negative section and $\ell$ is a fiber. Since $D$ is nef and big, $a >0$ and $b\ge ae$ for $\si^2=-e$ if $e > 0$. If $e=0$, then $a, b> 0$. If $D$ is not ample, then $a> 0$ and $b=ae$.
By the canonical bundle formula for a projective bundle, 
$$K_V=-2\Si-p^*((a+2)\si+(b+e+2)\ell).$$

Write $K_V=f^*K_W+c\Si+dp^*\si$. Let $\bar\ell:=p^*\ell|_\Si$. Then 
$$-2=K_\Si \cdot \bar\ell=(K_V+\Si) \cdot \bar\ell=-\Si \cdot \bar\ell-(a+2).$$
So $\Si \cdot \bar\ell=-a$, and $c=\frac{2-a}{a}$. We also get $d=\frac{2-e}{e}$ (resp. $d=0$) if $D$ is not ample (resp. ample).

We first treat the case when $D$ is ample. Then $d=0$, and we have 
$$f^*K_W = -\frac{a+2}{a}\Si-p^*((a+2)\si+(b+e+2)\ell).$$
Since we assume that $W$ is a $\QQ$-Gorenstein degeneration of $\PP^3$, $(-K_W)^3=64$. Let $A=p^*((a+2)\si+(b+e+2)\ell)$. Then we have
\begin{itemize}
    \item $\Si^3=(-a\si-b\ell)^2=-a^2e+2ab$,
    \item $\Si^2A=(-a\si-b\ell)((a+2)\si+(b+e+2)\ell)=(a^2+2a)e+(-2ab-ae-2a-2b)$,
    \item $\Si A^2=((a+2)\si+(b+e+2)\ell)^2=-(a^2+4a+4)e+(2ab+2ae+4a+4b+4e+8)$.
\end{itemize}
Therefore, 
\[ 
\begin{aligned} 
-K_W^3&= \prbra*{\frac{a+2}{a}}^3(-a^2e+2ab)
+ 3 \prbra*{\frac{a+2}{a}}^2((a^2+2a)e+(-2ab-ae-2a-2b)) \\
&\phantom{=} +3 \prbra*{\frac{a+2}{a}} (-(a^2+4a+4)e+(2ab+2ae+4a+4b+4e+8))\\
&=\frac{e}{a}(-(a+2)^3+3(a+2)^2)+\frac{2b}{a^2}((a+2)^3-3(a+2)^2)+\frac{6}{a}(a+2)^2 \\
&=\frac{(a+2)^2}{a^2}(-a^2e+2ab+ae-2b+6a).
% & &\eqqcolon(*).
\end{aligned} 
\]
Computation shows that if $a \le 2$, then $-K_W^3=64$ only when $a=b=2$, which is the type II.

Suppose that $e=0$. We may assume that $a \ge b$.
It is easy to verify that if $b \ge 3$ then $-K_W^3 >64$.  If $b=2$, then it can checked directly $-K_W^3=\frac{(a+2)^2}{a^2}(10a-4)> 64$ for all $a\ge 3$. Moreover, if $b=1$, then it can be checked again directly that $-K_W^3=\frac{(a+2)^2}{a^2}(8a-2)> 64$ for all $a\ge 4$, and $-K_W^3 < 64$ for $a=3, b=1$.

Suppose $e > 0$. Then we have
$$-K_W^3 \ge \frac{e}{a}((a+2)^3-3(a+2)^2)+\frac{6}{a}(a+2)^2$$
since $b\ge ae$. The above equality holds when $b=ae$, and we have again $-K_W^3>64$ if $a\ge 3$.

We next treat the case when $D$ is nef and big but not ample. In this case $a>0$, $e>0$ and $b=ae$.
We have 
$$f^*K_W = -\frac{a+2}{a}\Si-p^*((a+2)\si+(b+e+2)\ell) - \frac{2-e}{e} p^*\si.$$
Computation shows that 
$$-K_W^3=\frac{a+2}{ae}(a^2e^2+ae^2+e^2+6ae+12).$$
It is straightforward to verify that $-K_W^3=64$ only when $a=2$, $b=4$, and $e=2$, i.e., $W$ is of type III. 
\end{proof}

\section{Boundary divisors from the type II and moduli of double cover}\label{sec:Type_II}
The purpose of this section is to describe the pair of the type II degeneration of $\PP^3$ and hypersurfaces. Such study appeared in Horikawa's work on quintic surfaces already (cf. \cite{Hor75}).  
In \cite{Hor75}, Horikawa determines the structures of minimal algebraic surfaces $S$ with $p_g=4, q=0$, and $K^2=5$. There are three types of such surfaces. A surface of type I is birationally equivalent to a quintic surface in $\PP^3$ which has at most rational double points as its singularities. A surface of type IIa or of type IIb has a birational model which is a ramified double covering of a rational surface. The classes of surfaces of type I and of type IIa are, respectively,
closed under small deformations, and surfaces of type IIb are in the intersection of moduli of surfaces of type I and of type IIa. Surfaces of type IIb give a codimension one locus in the moduli of quintic surfaces. For surfaces of type I, the canonical bundle $K_S=\cO(1)$ is ample and has no base point. But for surfaces of type IIa (respectively IIb), $|K_S|$ has a one base point and it induces a double cover over $\PP^1\times\PP^1$ (respectively over a quadric cone).

\begin{Rmk} \label{odd in TYPE II}
DeVleming \cite[Section 5]{DeV22} describes one boundary divisor of the KSBA moduli of quintic surfaces using the type II degeneration of $\PP^3$. Let $W$ be a section of $\cO_Z(2)$ for $Z=\PP(1,1,1,1,2)$. Then a degenerating surface $B$ of quintics has a unique $\frac{1}{4}(1, 1)$ singularity at the vertex of $W$. 

This argument can be generalized to a degeneration of the pair $(\PP^3, \cO_{\PP^3}(d))$ for a positive odd integer with $d\ge 5$, i.e., the type II degeneration of $\PP^3$ produces a boundary divisor of the KSBA moduli of hypersurfaces, and a degenerating surface $B$ of hypersurfaces has a unique $\frac{1}{4}(1, 1)$ singularity at the vertex of $W$ \cite[Proposition 4.43]{DeV22}.
\end{Rmk}

\begin{Prop}\label{surface in type II} Let $Y$ be a type II degeneration of $\PP^3$. Consider a general element $(Y, B_{d, Y})$ of degeneration of the pairs $(\PP^3, B_d)$ where $B_d\in |\cO_{\PP^3}(d)|$.
\begin{itemize}
    \item If $d$ is odd with $d\ge 5$, then $B_{d, Y}$ has a unique $\frac{1}{4}(1, 1)$ singularity at the vertex of $Y$.
    \item If $d$ is even with $d\ge 4$, then $B_{d, Y}$ is a smooth complete intersection of multidegrees $(2, \frac{d}{2})$ in the weighted projective surface $\PP(1,1,1,1,2)$.
\end{itemize}
\end{Prop}
\begin{proof}
If $d$ is odd, then Remark~\ref{odd in TYPE II} gives the proof. If $d$ is even with $d\ge 4$, then a degenerating surface $B_{Y, d}$ is smooth because a general degree $d$ equation does not contain the vertex of $\PP(1,1,1,1,2)$. Let $[x_0:x_1: x_2: x_3: w]$ be a homogeneous coordinate of $\PP(1,1,1,1,2)$. We may assume that $Y$ is a cone over a smooth quadric surface in $\PP^3$ defined by $\sigma(x_0,x_1,x_2,x_3)$.  Then $B_{d, Y}$ is defined by 
\[ \left\{\begin{array}{l}
    \sigma(x_0, \ldots, x_3)=0; \\
    p(x_0, \ldots, x_3, w)=w^m+\sum_{i=1}^mw^{m-i}a_{2i}(x_0,\ldots, x_3)=0,
\end{array} \right.\]
where $d=2m$ and $a_{2i}$ is a general degree $2i$ equation in $\PP^3$ for $i=1,\ldots, m$.    

We can deform $B_{d, Y}$ by adding a constant times the variable $w$
\[ \left\{\begin{array}{l}
    tw-\sigma(x_0, \ldots, x_3)=0; \\
    p(x_0, \ldots, x_3, w)=0,
\end{array} \right.\]
where $t\in \CC$. For $t=0$ we have $B_{d, Y}$, whereas for $t\ne 0$ we can eliminate $w$ and obtain a hypersurface of degree $d$ in $\PP^3$ with the equation $p(x_0, \ldots, x_3, \sigma/t)=0$ (cf. \cite[Section 1]{CL10}).
\end{proof}

In the remaining part of this section, we generalize Soo's result for $d=10$ (cf. \cite{Soo24}) to any even integer $d\ge 10$. If $d=2d_1$ for odd $d_1$ then the construction and the computation are basically the same as Soo's results in \cite{Soo24}.

Let $V=\cO_T\oplus \cO_T(-2, -2)$ and $Y'=\PP(V)$ where $T=\PP^1\times\PP^1$, and let $\pi: Y'\to T$. Let $E$ be the $\pi$-section corresponding to the quotient line bundle $V \to \cO_T(-2, -2)\to 0$. Then there exists a smooth divisor $B_0\in |d_1E+\pi^*(\cO_T(2d_1, 2d_1))|$ such that $B_0\cap E=\emptyset$. Then, we have a double covering $p: X'\to Y'$ branched over $B:=B_0+E$ (resp. $B=B_0$) if $d_1$ is odd (resp. even). Then $X'$ is a smooth threefold with $({\rm Vol}, p_g)=(2(d_1-4)^3, \frac{(d_1-1)(d_1-2)(d_1-3)}{6})$, but not minimal. The canonical model $X$ of $X'$ is obtained by contraction of $p^{-1}(E)$. Let $\tau: Y'\to Y$ be the contraction of $E$ in $Y'$, then we have a double covering $X\to Y$. Here, if $d_1$ is odd, then the branch locus is the union of $\tau_*B_0$ and the isolated canonical singularity of $Y$; if $d_1$ is even, then the branch locus is $\tau_*B_0$.

\begin{Rmk}
\begin{enumerate}
    \item Since $Y$ is a canonical degeneration of $\PP^3$, $K_{\PP^3}=\cO_{\PP^3}(-4)$, and $K_{Y'}=-2E+\pi^*(\cO_T(-4, -4))$, we have $B_0=\frac{d_1}{2}(-K_{Y'})$, and 
    $\cO_{\PP^3}(2d_1)$ degenerates to $\tau_*B_0$. Therefore $X$ is a canonical degeneration of the double cover of $\PP^3$ branched over a smooth divisor in $|\cO_{\PP^3}(2d_1)|$.

    \item If $d_1$ is even, then the branch divisor $B$ does not have the component $E$. Then $X\to Y$ is a double covering branched over $\tau_*B_0$ and $X\to Y$ is \'etale over the isolated canonical singularity of $Y$.
\end{enumerate}
\end{Rmk}

The exact sequence of the tangent sheaves of the double cover gives
\[
\begin{tikzcd}
0\ar[r] &
\cT_{X'} \ar[r] &
p^*\cT_{Y'} \ar[r] &
\cT_{X'/Y'} \ar[r] & 0.
\end{tikzcd}
\]
From this, we can compute the cohomology spaces of the tangent sheaf $\cT_{X'}$.

\textit{Step 1.} Since $X'$ is of general type, $h^0(X', \cT_{X'})=0$.

\textit{Step 2.}
By \cite[Lemma 10]{Hor75}, we have $\cT_{X'/Y'} \cong p^*\cN_B$, where $\cN_B \cong \iota_* (\cO_Y(B)|_B)$ is the normal sheaf of the branch locus $B$ and $\iota: B \hookrightarrow Y'$ is the closed immersion. If $d_1$ is odd, then we have
\[\begin{array}{ll}
h^i(Y', \cO_{Y'}(B)) & = h^i(Y', \cO_{Y'}(d_1+1)\otimes \pi^*\cO_T(2d_1, 2d_1))\\
 & = h^i(T, {\rm Sym}^{d_1+1}(\cO_T\oplus\cO_T(-2, -2))\otimes \cO_T(2d_1, 2d_1))\\
& = h^i(T, \cO_T(-2, -2)\oplus \cO_T\oplus \cO_T(2, 2)\oplus\cdots\oplus \cO_T(2d_1, 2d_1))\\
 & = \begin{cases} \text{$1^2+3^2+\cdots + (2d_1+1)^2$, if $i=0$;}\\
\text{0, if $i=1$;}\\
\text{1, if $i=2$.}
\end{cases}
\end{array}\]
If $d_1$ is even then
\[\begin{array}{ll}
h^i(Y', \cO_{Y'}(B)) & = h^i(Y', \cO_{Y'}(d_1)\otimes \pi^*\cO_T(2d_1, 2d_1))\\
 & = h^i(T, {\rm Sym}^{d_1}(\cO_T\oplus\cO_T(-2, -2))\otimes \cO_T(2d_1, 2d_1))\\
& = h^i(T, \cO_T\oplus \cO_T(2, 2)\oplus\cdots\oplus \cO_T(2d_1, 2d_1))\\
 & = \begin{cases} \text{$1^2+3^2+\cdots + (2d_1+1)^2$, if $i=0$;}\\
\text{0, if $i=1$ or $i=2$.}
\end{cases}
\end{array}\]
Therefore, we have the following computation for $d_1$ being odd.
\[
\begin{tikzcd}
    & 0 \ar[r] & \cO_{Y'} \ar[r] & \cO_{Y'}(B) \ar[r] & \iota_* \iota^* \cO_{Y'}(B) \cong \cN_B \ar[r] & 0 \\[-4ex]
    h^0: &  & 1 & \sum_{m=0}^{d_1}(2m+1)^2 & \sum_{m=1}^{d_1}(2m+1)^2 \\[-4ex]
    h^1: &   & 0 &  0  & 0  \\[-4ex]
    h^2: &   & 0 &  1  & 1  \\[-4ex]
    h^3: &   & 0 &  0  &0   &  .
\end{tikzcd}
\]
We also have the following computation for $d_1$ being even.
\[
\begin{tikzcd}
    & 0 \ar[r] & \cO_{Y'} \ar[r] & \cO_{Y'}(B) \ar[r] & \iota_* \iota^* \cO_{Y'}(B) \cong \cN_B \ar[r] & 0 \\[-4ex]
    h^0: &  & 1 & \sum_{m=0}^{d_1}(2m+1)^2 & \sum_{m=1}^{d_1}(2m+1)^2 \\[-4ex]
    h^1: &   & 0 &  0  & 0  \\[-4ex]
    h^2: &   & 0 &  0  & 0  \\[-4ex]
    h^3: &   & 0 &  0  & 0  &  .
\end{tikzcd}
\]

\textit{Step 3.} We have $h^i(X', p^*\cT_{Y'})=h^i(Y', \cT_{Y'})+h^i(Y', \cT_{Y'} \otimes L^\vee)$ where $2L\sim B$. 

\textit{Step 3a (Computation of the cohomology of $\cT_{Y'}$).} By the description in \cite{Soo24}, $Y'$ is a complete toric variety whose fan has 6 rays. Their generators are described in the following table.

\begin{center}
    \begin{tabular}{c||c|c|c|c|c|c}
        generator of $\rho$ &
        $(1,0,2)$ & $(0,1,2)$ & 
        $(-1,0,0)$ & $(0,-1,0)$ &
        $(0,0,1)$ & $(0,0,-1)$ \\ \hline
        divisor class $[D_\rho]$ & 
        $\Sigma_0$ & $\Sigma_1$ &
        $\Sigma_0$ & $\Sigma_1$ &
        $E$ & $E + 2\Sigma_1 + 2\Sigma_0$ 
    \end{tabular}
\end{center}
Here, $\Sigma_0$ and $\Sigma_1$ are the pull back of $\cO_{\PP^1}(1)$ coming from two projections of $T\to\PP^1$. 

The toric Euler sequence then yields
\[
\begin{tikzcd}
    & 0 \ar[r] & 
    \cO_{Y'}^{\oplus 3} \ar[r] & 
    \displaystyle\bigoplus_{\rho \in \Sigma(1)} \cO_{Y'}(D_\rho) \ar[r] & 
    \cT_{Y'} \ar[r] & 0 \\[-4ex]
    h^0: &  & 3 & 19 & {16} \\[-4ex]
    h^1: &   & 0 &  0 & { 0} \\[-4ex]
    h^2: &   & 0 &  1 & { 1} \\[-4ex]
    h^3: &   & 0 &  0 & { 0} &  .
\end{tikzcd}
\]

\textit{Step 3b (Computation of the cohomology of $\cT_{Y'}\otimes L^\vee$).} On the other hand, the cohomologies $h^i(Y', \cT_{Y'}\otimes L^\vee)$ can be computed from the following exact sequence tensoring with $L^\vee$
\[ 
\begin{tikzcd}
    0 \ar[r]
    & \Theta_{Y'/T} \ar[r]
    & \cT_{Y'} \ar[r] 
    & \pi^* \cT_T \ar[r] & 0,
\end{tikzcd}
\] 
where $\pi^* \cT_T$ can be computed by pulling back the following exact sequence by $\pi$ and then tensoring with $L^\vee$: 
\[ 
\begin{tikzcd}
    0 \ar[r]
    & \Theta_{T/\PP^1} \ar[r]
    & \cT_{T} \ar[r] 
    & \mathrm{pr}^*\cT_{\PP^1} \ar[r] & 0.
\end{tikzcd}
\]
Notice that we have
$\Theta_{Y'/T} \cong \cO_{Y'}(2E + 2\Sigma_1 + 2\Sigma_0) = \cO_{Y'}(2) \otimes \pi^* \cO_T(2,2)$ and $\Theta_{T/\PP^1} \cong \cO_T(2,0)$ by considering the first Chern classes. Routine computations and the above exact sequences then yield
\[
h^i(Y', \Theta_{Y'/T}\otimes L^\vee) = 
\left\{
\begin{aligned}
    &\frac{(d_1 - 2)(d_1 - 3)(d_1 - 5)}{6},& &i = 3; \\
    &0, & &\text{otherwise.}
\end{aligned}
\right.
\] 
and 
\[
h^i(Y', \pi^*\cT_{T}\otimes L^\vee) = 
\left\{
\begin{aligned}
    &\frac{(d_1 - 2)(d_1 - 3)(d_1 - 5)}{3},& &i = 3; \\
    &0, & &\text{otherwise.}
\end{aligned}
\right.
\]
Thus, 
\[
h^i(Y', \cT_{Y'}\otimes L^\vee) = 
\left\{
\begin{aligned}
    &\frac{(d_1 - 2)(d_1 - 3)(d_1 - 5)}{2}& &i = 3; \\
    &0, & &\text{otherwise.}
\end{aligned}
\right.
\]

In conclusion, if $d_1$ is odd then we have the following table.
\[
\begin{tikzcd}
&[-3em] 0 \ar[r] 
&[-3em] \cT_{X'}      \ar[r] 
& p^*\cT_{Y'}   \ar[r] 
& p^*\cN_B        \ar[r] 
&[-3em] 0 \\[-4ex]
    h^0: &  & 0 & 16 & \frac{(2d_1 + 1)(2d_1 + 2)(2d_1 + 3)}{6}-1 \\[-4ex]
    h^1: &   & \frac{(2d_1 + 1)(2d_1 + 2)(2d_1 + 3)}{6} - 17 & 0 &   0 \\[-4ex]
    h^2: &   & k &  1  & 1 \\[-4ex]
    h^3: &   & k + \frac{(d_1 - 2)(d_1 - 3)(d_1 - 5)}{2} &  \frac{(d_1 - 2)(d_1 - 3)(d_1 - 5)}{2} &   0 &  ,
\end{tikzcd}
\] 
where $k$ could be $0$ or $1$.

On the other hand, if $d_1$ is even then we have the following table.
\[
\begin{tikzcd}
&[-3em] 0 \ar[r] 
&[-3em] \cT_{X'}      \ar[r] 
&[-1em] p^*\cT_{Y'}   \ar[r] 
&[-1em] p^*\cN_B        \ar[r] 
&[-3em] 0 \\[-4ex]
    h^0: &  & 0 & 16 & \frac{(2d_1 + 1)(2d_1 + 2)(2d_1 + 3)}{6}-1 \\[-4ex]
    h^1: &   & \frac{(2d_1 + 1)(2d_1 + 2)(2d_1 + 3)}{6} - 17 & 0 &   0 \\[-4ex]
    h^2: &   & 1 + k' &  1 + k'  & 0 \\[-4ex]
    h^3: &   & \frac{(d_1 - 2)(d_1 - 3)(d_1 - 5)}{2}  - 2 + k' &  \frac{(d_1 - 2)(d_1 - 3)(d_1 - 5)}{2} - 2 + k' &   0 &  .
\end{tikzcd}
\] 
where $k' = h^2(\cT_{Y'}\otimes L^\vee)$ can possibly be $0$, $1$, or $2$. 

\begin{Prop} \label{prop:h2(TX)}
We have $H^2(\cT_{X'})=0$ if $d_1$ is odd, and $H^2(\cT_{X'})=1$ if $d_1$ is even.
\end{Prop} 
\begin{proof} Since $H^1(p^*\cN_B)=0$, we have an injection from $H^2(\cT_{X'})$ to $H^2(p^*\cT_{Y'})$. Since we have $H^2(\cT_{Y'}\otimes L^\vee)=0$, there is no anti-invariant part in $H^2(\cT_{X'})$. It suffices to understand the map from $H^2(\cT_{Y'})$ to $H^2(\cN_B)$. 

If $d_1$ is even then $H^2(\cN_B)=0$, therefore we have $H^2(\cT_{X'})=1$.
% to compute $H^2(T_{X'})$. We have $H^2(T_{Y'})\cong H^2(Y', \cO_{Y'}(E))$ from the toric Euler short exact sequence and $H^2(Y', \cO_{Y'}(E))\cong H^2(Y',N_B)$. And both cohomologies are one-dimensional, which correspond naturally to $H^2(\cO_T(-2, -2))$. 

Now, suppose that $d_1$ is odd.
We claim that this map is an isomorphism, from which $k=0$ follows. Consider the following commutative diagram
\[
\begin{tikzcd}
    H^2(Y', \cT_{Y'}) \ar[r]
    & H^2(Y, \cN_B) \cong H^2(B, \cO_B(B)) \\
    H^2(Y', E) \ar[u, "\sim", sloped] \ar[r, "\sim"] 
    % \ar[ru, dashed]
    & H^2(Y', E + B_0), \ar[u, "\sim", sloped] 
\end{tikzcd}
\]
where the map on the left comes from the toric Euler sequence and the map on the right is the restriction.
On the bottom row, both $H^2(Y', E)$ and $H^2(Y', E + B_0)$ are isomorphic to the one-dimensional $H^2(T, \cO_T(-2,-2))$ via the Leray spectral sequence, and this is compatible with the map between them (induced by $- \otimes \cO_{Y'}(B_0)$), so the bottom map is an isomorphism. This yields the desired isomorphism on the top.
\end{proof}

\begin{Thm}
Type II gives a boundary divisor in the moduli space $\cM_d^\mathrm{can}$ and smoothness of $\cM_d^\mathrm{sm}$ extends to boundary. 
\end{Thm}
\begin{proof}
    If $d_1$ is odd then $H^2(\cT_{X'})=0$ implies directly the theorem. 
    
    Suppose $d_1$ is even. Even though $h^2(\cT_{X'})=1$, all deformations of $X'$ comes from $H^0(\cN_B)$ since the map from $H^0(\cN_B)$ to $H^1(\cT_{X'})$ is surjective. Since $H^1(\cN_B)=0$, this deformation has no obstruction.
\end{proof}

Through the double covering corresponding, we get the following corollary. This corollary is also mentioned in \cite[Proposition 3.10]{DeV22}. It can be also proven by using \cite[Lemma 5.4]{ADL23}, the vanishing $H^1(\cO_Y(B))=0$, and \cite[Proposition 3.3]{Has99}, because $H^2(\cT_Y)=0$ and $H^1(\cO_Y(B))=0$ imply $H^2(\cT_{(Y, B)})=0$, i.e., the obstruction space for the deformation of the pair is zero.

\begin{Cor}
For an even $d$, the moduli space $\overline{\cN}_{d, 4}$ is smooth at $[(Y, B_{d, Y})]$ if $Y$ is the type II, they give a a boundary divisor of $\overline{\cN}_{d, 4}$.
\end{Cor}

\section{Boundary from the type III and moduli of double cover}\label{sec:type_III}
The purpose of this section is to describe the pair of the type III degeneration of $\PP^3$ and hypersurfaces.
We have the same description for $(Y, B_{d, Y})$ as Proposition~\ref{surface in type II} if it is the type III degeneration of the pairs $(\PP^3, B_d)$ where $B_d\in |\cO_{\PP^3}(d)|$. 

In this section, we also generalize Soo's result for $d=10$ (cf. \cite{Soo24}) to any even integer $d\ge 10$.

Let $V=\cO_T\oplus \cO_T(-2, -4)$ and $Y'=\PP(V)$ where $T= \FF_2$, and let $\pi: Y'\to T$. Here, $\cO_T(a, b)$ means $\cO(a\si+b\ell)$, where $\si$ is the infinite section and $\ell$ is a fiber in $T$. Let $E$ be the $\pi$-section corresponding to the quotient line bundle $V \to \cO_T(-2, -4)\to 0$. Then there exists a smooth divisor $B_0\in |d_1 E + \pi^*(\cO_T(2d_1, 4d_1))|$ such that $B_0\cap E=\emptyset$. Then, we have a double covering $p: X'\to Y'$ branched over $B \coloneqq B_0 + E$ (resp. $B=B_0$) if $d_1$ is odd (resp. even). Then $X'$ is a smooth threefold with $({\rm Vol}, p_g)=(2(d_1 - 4)^3, \frac{(d_1 - 1)(d_1 - 2)(d_1 - 3)}{6})$, but not minimal. 

The same steps 1, 2, and 3a in Section~\ref{sec:Type_II} can also be carried out in this case of Type III. In step 3a, one uses the following toric data for $Y'$ of type III:
\begin{center}
    \begin{tabular}{c||c|c|c|c|c|c}
        generator of $\rho$ &
        $(1,0,4)$ & $(0,1,2)$ & 
        $(-1,2,0)$ & $(0,-1,0)$ &
        $(0,0,1)$ & $(0,0,-1)$ \\ \hline
        divisor class $[D_\rho]$ & 
        $\Sigma_0$ & $\Sigma_1$ &
        $\Sigma_0$ & $\Sigma_1 + 2\Sigma_0$ &
        $E$ & $E + 2\Sigma_1 + 4\Sigma_0$ 
    \end{tabular}
\end{center}
Here, $\Sigma_0$ is the pull back of a fiber in $T$ and $\Sigma_1$ is the pull back of the negative section in $T$.

However, step 3b is more problematic since there are many nonzero cohomologies.
We can still compute $H^1(\cT_{Y'}\otimes L^\vee)$ by tensoring the toric Euler sequence with $L^\vee$:
\[
\begin{tikzcd}
    & 0 \ar[r] & 
    (\cO_{Y'} \otimes L^\vee)^{\oplus 3} \ar[r] & 
    \displaystyle\bigoplus_{\rho \in \Sigma(1)} \cO_{Y'}(D_\rho) \otimes L^\vee \ar[r] & 
    \cT_{Y'} \otimes L^\vee \ar[r] & 0 \\[-4ex]
\end{tikzcd}
\]
% If $d_1$ is odd then 
% $$H^1(\cT_{Y'}\otimes L^\vee)\cong H^1(D_\rho-L)\cong H^2(K_{Y'}+L-D_\rho)^\vee\cong H^2(T, K_T)^\vee\cong H^0(\cO_T)=\CC$$ for $D_\rho=\pi^*(\cO_T(1, 2))$.
Carrying out the computations, we can infer the following proposition.
\begin{Prop}
    We have
    \begin{itemize}
    \item $h^1(\cT_{Y'})=1$.
    \item $h^1(\cT_{Y'}\otimes L^\vee)=1$ and $h^1(\cT_{X'}) = \frac{(2d_1 + 1)(2d_1 + 2)(2d_1 + 3)}{6} - 16$ if $d_1$ is odd. 
    \item $h^1(\cT_{Y'}\otimes L^\vee)=0$ and $h^1(\cT_{X'}) = \frac{(2d_1 + 1)(2d_1 + 2)(2d_1 + 3)}{6} - 17$ if $d_1$ is even.
    % \item $h^2(\cT_{X'}) - h^3(\cT_{X'})
    % = - \frac{1}{2} \left(
    % d_1^3 - 10 d_1^2 + 31 d_1 - 32 \right)$.
    \end{itemize}
    This yields the following table if $d_1$ is odd:
    \[
\begin{tikzcd}
&[-3em] 0 \ar[r] 
&[-3em] \cT_{X'}      \ar[r] 
& p^*\cT_{Y'}   \ar[r] 
& p^*\cN_B        \ar[r] 
&[-3em] 0 \\[-4ex]
    h^0: &  & 0 & 17 & \frac{(2d_1 + 1)(2d_1 + 2)(2d_1 + 3)}{6}-1 \\[-4ex]
    h^1: &   & \frac{(2d_1 + 1)(2d_1 + 2)(2d_1 + 3)}{6} - 16 & 2 &   0 \\[-4ex]
    h^2: &   & * & * & 1 \\[-4ex]
    h^3: &   & * & * &   0. &  
\end{tikzcd}
\] 
and the following table if $d_1$ is even:
\[
\begin{tikzcd}
&[-3em] 0 \ar[r] 
&[-3em] \cT_{X'}      \ar[r] 
& p^*\cT_{Y'}   \ar[r] 
& p^*\cN_B        \ar[r] 
&[-3em] 0 \\[-4ex]
    h^0: &  & 0 & 17 & \frac{(2d_1 + 1)(2d_1 + 2)(2d_1 + 3)}{6}-1 \\[-4ex]
    h^1: &   & \frac{(2d_1 + 1)(2d_1 + 2)(2d_1 + 3)}{6} - 17 & 1 &   0 \\[-4ex]
    h^2: &   & * & * & 1 \\[-4ex]
    h^3: &   & * & * &   0. &  
\end{tikzcd}
\] 
\end{Prop}

\begin{Rmk}
    By considering the toric Euler sequence, we have 
    $$H^1(\cT_{Y'})\cong H^1(\cO_{Y'}(\pi^*\cO_{\mathbb F_2}(1, 0)))\cong H^1(\cT_{\mathbb F_2}).$$
    Since $H^2(\cT_{\mathbb F_2})=0$, i.e. no obstruction of the deformation of $\mathbb F_2$, the obstruction map from $H^1(\cT_{Y'})$ to $H^2(\cT_{Y'})$ is zero even though $h^2(\cT_{Y'})=1$
\end{Rmk}

We will call a threefold is Type III if it is a double cover of the type III threefold $Y$ branched over $B_{d, Y}$.

\begin{Thm}
    If $d_1$ is even then the moduli space $\cM_d^\mathrm{can}$ at Type III is smooth. 
\end{Thm}

\begin{Rmk}
    If $d_1$ is odd then we expect $\cM_d^\mathrm{can}$ is singular at Type III, i.e. the obstruction map from $H^1(\cT_{Y'}\otimes L^\vee)$ to $H^2(\cT_{Y'}\otimes L^\vee)$ is nonzero. If $d_1=5$ then the moduli space at Type III is singular by \cite[Theorem 1.3]{CHJ24}, because there is no extra threefold with $p_g=4$ and ${\rm Vol}=2$ which is not a double cover.
\end{Rmk}

\begin{Rmk}
 In $d_1 = 5$ case, we can actually show that $h^2(\cT_{X'}) = 2$ as follows. Notice that $H^2(p^*\cT_{Y'}) = H^2(\cT_{Y'}) \oplus H^2(\cT_{Y'}\otimes L^\vee)$, and the invariant part $H^2(\cT_{Y'})$ maps isomorphically onto $H^2(\cN_B)$ by the same argument as Proposition \ref{prop:h2(TX)}. We thus see that $H^2(\cT_{X'})$ is isomorphic to the anti-invariant part $H^2(\cT_{Y'}\otimes L^\vee)$, which has dimension $2$.
    
In the general case $d_1 \geq 7$, we still have $H^2(\cT_{X'}) \cong H^2(\cT_{Y'}\otimes L^\vee)$ by the same argument; however, the dimension of the latter space is more difficult to compute.
\end{Rmk}

Even though we cannot prove the moduli space $\cM_d^\mathrm{can}$ at Type III is singular when $d_1$ is odd and $d_1\ge 7$. But the following corollary can be proven using \cite[Lemma 5.4]{ADL23}, the vanishing $H^1(\cO_Y(B))=0$, and \cite[Proposition 3.3]{Has99}.

\begin{Cor}
For an even $d$, the moduli space $\overline{\cN}_{d, 4}$ is smooth at $[(Y, B_{d, Y})]$ if $Y$ is the type III, they give a codimension 2 boundary in $\overline{\cN}_{d, 4}$.
\end{Cor}

\section{Boundary divisors from the type IV}\label{sec:type IV}

We begin with the description of type IV. Let $Y_0$ be the $\PP(1,2,3)$-bundle over $\PP^1$ as in \cite{ADL23}. One has $-K_{Y_0}=6\Sigma + 12 F$, where $\Sigma $ is the tautological divisor $\cO(1)$ and $F$ is the fiber. Fiberwise, we have that $\Sigma$ is the divisor defined by $(x=0)$ in $\PP(1,2,3)$, where $x$ is the coordinate of $\PP(1,2,3)$ of weight $1$. There exist singularities at $[0:1:0]$ and $[0:0:1]$ of type $\frac{1}{2}(1,1)$ and $\frac{1}{3}(1,2)$ respectively, fiberwise.

As in \cite{ADL23}, we consider $\pi_1: Y_1 \to Y_0$  the $(1,2)$ weighted blowup along the smooth curve which intersects the point $[0:1:1]$ in $\Sigma$. Let $D_1$ be the exceptional divisor and $\Sigma_1$ be the proper transform of $\Sigma$ on $Y_1$. Then 
\( -K_{Y_1}= 6 \Sigma_1+4D_1+12 F.\)
The map to its canonical model $\psi:Y_1 \to Y$ contracts $\Sigma_1$ to a point. The relative canonical model $\pi_{\flat}: Y_1 \to Y_\flat$ over $\PP^1$ contracts $\Sigma_1\cap F$ to a $D_5$ singularity.

Clearly, there exists a crepant resolution of  singularities ${\pi}: \tilde{Y} \to Y_1$. We have $\tilde{E_1},\tilde{E_2}$ over the $\frac{1}{3}(1,2)$ singularity with $\tilde{E_2}$ intersecting $\tilde{\Sigma}$, $\tilde{E_3}$ over the original $\frac{1}{2}(1,1)$, and $\tilde{E_4}$ over the $\frac{1}{2}(1,1)$ singularity at $\Sigma_1 \cap D_1$. 

Then 
\[ -K_{\tilde{Y}}= 12 F + 4\tilde{D} + 6 \tilde{\Sigma}+2 \tilde{E_1}+ 4 \tilde{E_2} + 3 \tilde{E_3}+ 5 \tilde{E_4}\]
where $\tilde D$ (resp. $\tilde\Sigma$) is the proper transform of $D_1$ (resp. $\Sigma_1$).

Let $\tau: \tilde{Y} \to Y'$ be the contraction that contracts $\tilde{E_2}, \tilde{E_3}, \tilde{E_4}$ and $\tilde{\Sigma}$. Then $Y'$ is a crepant partial resolution of $Y_\flat$ so that each fiber is a partial resolution of the $D_5$ singularity to a $D_4$ singularity. 
We summarize the above description of birational models into the following diagram. Note that $\pi, \pi_\flat, \tau, \tau_\flat$ are crepant.

\begin{equation}\label{comp}
\begin{tikzcd}
    & \tilde{Y} 
    \ar[dl,"\pi"']
    \ar[dr, "\tau"] & \\
    Y_1 
    \ar[dr, "\pi_\flat"]
    \ar[d, "\pi_1"']
    \ar[ddr, "\psi"'] &  & 
    Y'
    \ar[dl, "\tau_\flat"']
    \ar[d, "\varphi"]
    \ar[ddl, "\phi"] \\[+2ex]
    Y_0 \ar[d] & Y_\flat \ar[d] & \mathbb{F}_4 \\[+2ex]
    \mathbb{P}^1 & Y & 
\end{tikzcd}
\end{equation}
% \begin{eqnarray}\label{comp}
% \xymatrix{
% & \tilde{Y} \ar[dl]_{\pi}\ar[dr]^{\tau} & \\
% Y_1 \ar[dr]^{\pi_\flat}\ar[d]_{\pi_1}\ar[ddr]_{\psi} &  & Y' \ar[dl]_{\tau_\flat}\ar[d]^{\varphi} \ar[ddl]^{\phi}\\
% Y_0 \ar[d] & Y_\flat \ar[d] & \mathbb{F}_4 \\
% \mathbb{P}^1 & Y & 
% }
% \end{eqnarray}

Note also that 
\[ \tau^*D'=\tilde{D}+\frac{1}{2}\tilde{E_2}+\tilde{\Sigma}+\frac{1}{2}\tilde{E_3}+\tilde{E_4},\] and 
\[ \tau^*E'_1= \tilde{E_1}+\tilde{E_2}+\tilde{\Sigma}+ \frac{1}{2}\tilde{E_3}+\frac{1}{2}\tilde{E_4}.\] Therefore, 
\[ -K_{Y'}=12F+4D'+2E'_1. \]  We have $\psi^*H = 3F+D_1 +\frac{3}{2}\Sigma_1$ on $Y_1$ and $\phi^*H=3F+D'+\frac{1}{2}E'_1$ on $Y'$, where $H$ is the hyperplane section on $Y$.

\begin{Prop}\label{type IV} Let $Y$ be the type IV degeneration of $\PP^3$. Let $(Y, B_{d, Y})$ be a $\QQ$-Gorenstein degeneration of the pairs $(\PP^3, B_d)$ with $dK_Y+4B_{d, Y}\sim 0$. Then
\begin{itemize}
    \item If $d$ is odd with $d\ge 5$, then it does not occur \cite[Theorem 4.35]{DeV22}.
    \item If $d$ is even with $d\ge 4$ and $d\notin 4\ZZ$, then it does not occur.
    \item If $d$ is even with $d\ge 4$ and $d\in 4\ZZ$, then $B_{d, Y}$ is a smooth surface satisfying the following:
    \begin{itemize}
    \item Consider a general fiber $\tilde S$ of a weak Del Pezzo surface of degree 4 fibration $\tilde Y\to\PP^1$ in the above description~(\ref{comp}). Then the curve $C=B_{d, Y}\cdot \tilde S$ intersects $\frac{d}{2}$ number of points with multiplicity at each fiber of the conic fibration of $\tilde S$, and $C$ does not meet the $D_5$-configuration in $\tilde S$.
    \item $B_{d, Y}$ does not contain the unique singularity $p$ of $Y$.
    \item $B_{d, Y}$ is composed of a pencil of curves.
    \end{itemize}
\end{itemize}
\end{Prop}
\begin{proof} If $d$ is odd with $d\ge 5$, then the proof is obtained by \cite[Theorem 4.35]{DeV22}. 

Suppose $d$ is even. Then we prove by using the argument in the proof of Theorem 4.35 in \cite{DeV22} and the explicit description of $Y$ in \cite{ADL23} and \cite{HP24}. Note that our $\tilde Y$ in the above description~(\ref{comp}) is $X'$ in the proof of Theorem 4.35 in \cite{DeV22}, and we have a weak Del Pezzo surface fibration $\tilde Y\to\PP^1$. Then, due to the proof of Case 2  in \cite[p. 1362]{DeV22}, $B_{d, Y}$ does not contain the unique singularity $p$ of $Y$. Therefore, $C$ does not meet the $D_5$-configuration in $\tilde S$. 

Let $f:\tilde Y\to Y$ be a resolution in Section 2. Since $dK_Y+4B_{d, Y}\sim 0$ and $K_{\tilde Y}=f^*K_Y$, we still have 
$$dK_{\tilde Y}+4B_{d, \tilde Y}\equiv 0.$$ 
We also have $B_{d, \tilde Y}=B_{d, Y}$ because $B_{d, Y}$ does not contain the unique singularity $p$ of $Y$. Notice that $K_{\tilde Y}\cdot(\text{conic fiber})=-2$, so $C=B_{d, Y}\cdot \tilde S$ intersects $\frac{d}{2}$ number of points with multiplicity at each fiber of the conic fibration of $\tilde S$. But the conic fibration of a weak Del Pezzo surface with $D_5$ configuration has one double line ($2\tilde D\cdot \tilde S$ in our notation) in the conic fibration. Therefore, if $d\not\in 4\ZZ$, then it cannot occur because the intersection number between $C$ and this line is $\frac{d}{2}$, which is odd.

Finally, if $d\in 4\ZZ$, then a weak Del Pezzo surface fibration $\tilde Y\to\PP^1$ implies that $B_{d, Y}$ is composed of a pencil of curves since $B_{d, Y}$ does not meet the $cD_5$ singular curve in $Y_\flat$ (the above description~(\ref{comp})).
\end{proof}

\begin{Rmk}
If $d=4$, then a general quartic K3 surface can degenerate to an elliptic K3 surface. Similar phenomena can happen when $d\in 4\ZZ$.
\end{Rmk}

\begin{Thm}
    The moduli space $\overline{\cN}_{(d, 4)}$ is smooth at the pair $(Y, B_{d, Y})$ if $Y$ is type IV.
\end{Thm}
\begin{proof}
    By Proposition~\ref{type IV}, we only need to consider the case when $d\in 4\ZZ$. Again, by Proposition~\ref{type IV}, $B_{d, Y}$ is smooth and does not contain the unique singular point $p$ of $Y$. Let $f: \tilde Y\to Y$, then we have
    $$
    H^1(\tilde Y, B_{d, \tilde Y})
    = H^1\prbra*{\tilde Y, K_{\tilde Y}+\prbra*{\frac{d}{4}+1}(-K_{\tilde Y})
    }=\text{0 for $i>0$}
    $$
    by the Kawamata-Viehweg vanishing theorem since $-K_{\tilde Y}=f^*(-K_{Y})$ is nef and big. This implies that $H^1(\cN_{B_{d, \tilde Y}})=0$ because $H^2(\cO_{\tilde Y})=0$. 
    The theorem then follows from \cite[Lemma 4.15]{ADL23} and \cite[Proposition 3.3]{Has99}
\end{proof}

\begin{Cor}\label{surface in type IV}
    If $d\in 4\ZZ$, then there is a smooth boundary divisor in $\overline{\cN}_{d, 4}$ whose general element $B_{d, Y}$ is a smooth surface with a pencil structure.
\end{Cor}

\end{document}